\documentclass[reqno,12pt]{article}
\usepackage{amsmath,amssymb,amsthm,mathtools}
\usepackage{mathrsfs}
\usepackage{bm}
\usepackage{amsmath}
\usepackage{tikz-cd}
\usepackage{booktabs}
\usepackage{xcolor}
\usepackage{hyperref}
\usepackage{float}
\usepackage{orcidlink}
\usepackage{comment}
\hypersetup{
    colorlinks=true,
    linkcolor=blue,
    citecolor=blue,
    urlcolor=blue}
\usepackage[margin= 0.8in]{geometry}
\newtheorem{theorem}{Theorem}[section]
\newtheorem{lemma}[theorem]{Lemma}
\newtheorem{proposition}[theorem]{Proposition}
\newtheorem{corollary}[theorem]{Corollary}

\theoremstyle{definition}

\allowdisplaybreaks

\newcommand{\Z}{\mathbb{Z}}

\begin{document}
\title{On the infinite sum of reciprocals of the fourth powers of balancing numbers}
\author{
 Subhasis Panda\\
\small Department of Mathematics, School of Advanced Sciences\\
\small VIT-AP University, Amravati-522241, Andhra Pradesh, India\\
\small {subhasispanda559@gmail.com}\\\\
Aditya Kumar Dash\\ 
\small Department of Mathematics and Applied Statistics, School of Applied Sciences \\
\small KIIT Deemed to be University, Bhubaneswar-751024, India\\
\small {adityadash44@gmail.com}\\\\
Utkal Keshari Dutta\\ 
\small Department of Mathematics and Applied Statistics, School of Applied Sciences \\
\small KIIT Deemed to be University, Bhubaneswar-751024, India\\
\small {utkal.duttafma@kiit.ac.in}
}
\date{}
\maketitle

\begin{abstract} 
In this note, we study the infinite reciprocal sum $\sum_{k=n}^{\infty}1/B_k^4$ involving the fourth powers of balancing numbers $B_n$. We show that, for every $n\geq2$,
\begin{equation*}
	\left\lfloor
	\left(
	\sum_{k=n}^{\infty}\frac{1}{B_k^4}
	\right)^{-1}
	\right\rfloor
	=
	B_n^4-B_{n-1}^4
	-\left\lceil\frac{B_{2n-1}}{280}\right\rceil
	+\varepsilon_n,
\end{equation*}
where $\varepsilon_n=1$ if $n\equiv1\pmod{12}$ and $\varepsilon_n=0$ otherwise. This result extends the corresponding reciprocal-sum result for Fibonacci numbers due to Hwang, Park and Song to balancing numbers. 
\end{abstract}



\section{Introduction}
A natural number $n$ is called a balancing number if it is a solution of the Diophantine equation
$
1+2+\dots + (n-1) = (n+1) + (n+2) + \dots +(n+l),
$
where $l$ is called the balancer corresponding to $n$ \cite{BGKP}.
Let $\{B_{n}\}_{n \geq 0}$ denote the balancing sequence, defined recursively by $B_0=0$, $B_1=1$, and
 $B_{n} = 6B_{n-1}-B_{n-2}$ for $n \geq 2$. 
The positive square roots of $8B_n^2+1$ generate Lucas-balancing sequence $\{C_n\}_{n \geq 0}$. 
 The Lucas-balancing sequence satisfies the same recurrence as the balancing sequence, with different initial values, namely, $C_{n} = 6C_{n-1}-C_{n-2}$ for $n \geq 2$ where $C_{0}=1$ and $C_{1}=3$ \cite{GKP}. 
The closed-form expressions for these sequences are given by
\begin{equation*}
	B_n=\frac{\lambda^n-\mu^n}{\lambda-\mu},
	\qquad
	C_n=\frac{\lambda^n+\mu^n}{2},
\end{equation*}
where $\lambda=3+2\sqrt{2}$ and $\mu=3-2\sqrt{2}$ are the roots of the characteristic equation $x^2-6x+1=0$.

Many researchers have studied the partial sums of reciprocal Fibonacci and other related numbers. Ohtsuka and Nakamura \cite{HOSN} studied the partial sums of reciprocal Fibonacci numbers and obtained the following results, where $\lfloor\cdot\rfloor$ denotes the floor function.
For all $n \geq 2$, we have
\begin{equation*}
	\left\lfloor
	\left(\sum_{k=n}^{\infty}\frac{1}{F_k}\right)^{-1}
	\right\rfloor
	=
	\begin{cases}
		F_{n-2}, & \text{if $n$ is even},\\
		F_{n-2}-1, & \text{if $n$ is odd}.
	\end{cases}
\end{equation*}
and for all $n\geq1$,
\begin{equation*}
	\left\lfloor
	\left(\sum_{k=n}^{\infty}\frac{1}{F_k^2}\right)^{-1}
	\right\rfloor
	=
	\begin{cases}
		F_nF_{n-1}-1, & \text{if $n$ is even},\\
		F_nF_{n-1}, & \text{if $n$ is odd}.
	\end{cases}
\end{equation*}
Later, Holiday and Komatsu \cite{HTK} established several identities for generalized Fibonacci numbers $G_n$, defined by $G_n=a \,G_{n-1}+G_{n-2},\, n\geq2$, with initial values $(G_0,G_1)=(0,1)$.

 Zhang and Wang \cite{Zhang2011, Zhang2012} considered infinite sums involving reciprocals of Pell and squared Pell numbers and obtained several identities for these sums.  Xu and Wang \cite{Xu2013} evaluated the integer part of the infinite sums of reciprocal cubes of Pell numbers. Komatsu and Panda \cite{Komatsu2017} derived several identities involving infinite sums of reciprocals of balancing and Lucas-balancing numbers. Some of their results are as follows:
 \begin{equation*}
 	\left\lfloor
 	\left(
 	\sum_{k=n}^{\infty}\frac{1}{B_{lk}}
 	\right)^{-1}
 	\right\rfloor
 	=
 	B_{ln}-B_{l(n-1)}-1,\qquad n\geq1,
 \end{equation*}
 \quad  \quad  \quad  \quad  \quad  \quad  \quad  \quad  \quad  \quad  \quad  \quad  \quad  \quad  \quad  \quad and
 \begin{equation*}
 	\left\lfloor
 	\left(
 	\sum_{k=n}^{\infty}\frac{(-1)^k}{B_k}
 	\right)^{-1}
 	\right\rfloor
 	=
 	\begin{cases}
 		B_n+B_{n-1}, & \text{if $n$ is even},\\
 		B_n+B_{n-1}+1, & \text{if $n$ is odd}.
 	\end{cases}
 \end{equation*}
Here, $l$ is a fixed positive integer. Recently, Hwang et al. \cite{HPS} considered infinite sums involving reciprocals of the fourth powers of Fibonacci numbers and proved the following result.
\begin{equation*}
\left\lfloor
\left(
\sum_{k=n}^{\infty}\frac{1}{F_k^4}
\right)^{-1}
\right\rfloor
=
F_n^4-F_{n-1}^4-\frac{2{(-1)}^n}{5}F_{2n-1}-\left\{ \frac{n+2}{5} \right\}.
\end{equation*}

In the present study, we consider the infinite sums of reciprocals of fourth power of balancing numbers and derive the following result
\noindent
\begin{theorem}{\normalfont (Main theorem)}\label{A}
	\normalfont
	For every integer $n\geq 2$, define
	\begin{equation*}
		g_n:=B_n^4-B_{n-1}^4
		-\left\lceil\frac{B_{2n-1}}{280}\right\rceil
		+\varepsilon_n,
	\end{equation*}
	where
	\begin{equation*}
		\varepsilon_n=
	\begin{cases}
		1, & \text{if } B_{2n-1}\equiv 1 \pmod{280},\\
		0, & \text{otherwise}.
	\end{cases}	
	\end{equation*}
	Then
	\begin{equation*}
		\left\lfloor
		\left(
		\sum_{k=n}^{\infty}\frac{1}{B_k^4}
		\right)^{-1}
		\right\rfloor
		=g_n.
	\end{equation*}
\end{theorem}
Note that the condition $B_{2n-1}\equiv1\pmod{280}$
holds precisely when  $n\equiv1\pmod{12}$.  

The structure of the paper is as follows. In section~2, we present the identities and auxiliary results on balancing and Lucas-balancing numbers needed for the proof, including the periodicity of $B_{2n-1} \bmod 280$. In section~3, we establish the sandwiching inequality. In section~4, we use this inequality to prove Theorem~\ref{A}.

\section{Balancing and Lucas Balancing Numbers}
In this section, we discuss several identities and auxiliary results regarding balancing and Lucas-balancing numbers that will be used in the proof of the main theorem.
Throughout the paper, let $\lambda=3+2\sqrt{2}$ 
\text{and} 
$\mu=3-2\sqrt{2},$
so that $ \lambda\mu=1,\lambda+\mu=6,
\lambda-\mu=4\sqrt{2}$.\\

The following result can be found in \cite[Theorem~2.1]{GKP}.
\begin{lemma}\label{L1}
\normalfont For all integers $m \ge r \ge 0$,
\begin{equation*}
B_{m+r}B_{m-r}=B_m^2-B_r^2.
\end{equation*}
\end{lemma}
The next identity follows immediately from Lemma~\ref{L1} by setting $m=n+1$, $r=n$, and using $B_1=1$.

\begin{lemma}\label{L2}
\normalfont For all integers $n\ge1$,
\begin{equation*}
B_{n+1}^2-B_n^2=B_{2n+1}.
\end{equation*}
\end{lemma}

\begin{lemma}\label{L3}
\normalfont For all integers $n\ge1$,
\begin{equation*}
B_n^2+B_{n-1}^2=\frac{3C_{2n-1}-1}{8}.
\end{equation*}
\end{lemma}

\begin{proof}
Using the Binet formulas, we have
\begin{equation*}
\begin{aligned}
B_n^2+B_{n-1}^2
&=
\frac{\lambda^{2n}+\mu^{2n}-2
+\lambda^{2n-2}+\mu^{2n-2}-2}
{(\lambda-\mu)^2}  \\
&=
\frac{\lambda^{2n-2}(\lambda^2+1)
+\mu^{2n-2}(\mu^2+1)-4}
{(\lambda-\mu)^2}.
\end{aligned}
\end{equation*}
Since
$\lambda^2+1=\lambda(\lambda+\mu)=6\lambda$, and $\mu^2+1=\mu(\lambda+\mu)=6\mu,$ it follows that
\begin{equation*}
\begin{aligned}
B_n^2+B_{n-1}^2
&=
\frac{6(\lambda^{2n-1}+\mu^{2n-1})-4}
{(\lambda-\mu)^2} \\
&=
\frac{12C_{2n-1}-4}{32}
=\frac{3C_{2n-1}-1}{8}.
\end{aligned}
\end{equation*}
\end{proof}

\begin{corollary}\label{C1}
\normalfont For all integers $n\ge1$, we have
\begin{equation*}
B_n^2=\frac{8B_{2n-1}+3C_{2n-1}-1}{16} \,\, \text{and} \,\, B_{n-1}^2=\frac{3C_{2n-1}-8B_{2n-1}-1}{16}.
\end{equation*}
and consequently,
\begin{equation*}
B_n^4-B_{n-1}^4
=\frac{B_{2n-1}\bigl(3C_{2n-1}-1\bigr)}{8}.
\end{equation*}
\end{corollary}

\begin{proof}
The first two identities follow immediately by adding and subtracting the identities in Lemma~\ref{L2} and \ref{L3}. Furthermore 
\begin{equation*}
\begin{aligned}
B_n^4-B_{n-1}^4
&=(B_n^2-B_{n-1}^2)(B_n^2+B_{n-1}^2)\\
&=\frac{B_{2n-1}(3C_{2n-1}-1)}{8}.
\end{aligned}
\end{equation*}

\end{proof}

The balancing and Lucas-balancing numbers satisfy the addition formulas $B_{p+q}=B_pC_q+C_pB_q$ and
$C_{p+q}=C_pC_q+8B_pB_q$ \cite[Theorem~2.5]{GKP}. In particular, $C_{p+1}=3C_p+8B_p$. Now, taking $p=2n-1$ and $q=2$ in the addition formulas, and using $B_2=6$ and $C_2=17$, we obtain
\begin{equation}\label{eq1}
	B_{2n+1}=17B_{2n-1}+6C_{2n-1} \,\, \text{and} \,\, C_{2n+1}=48B_{2n-1}+17C_{2n-1}.
\end{equation}

\begin{lemma}\label{L5}
\normalfont For all integers $n\ge0$,  $C_n^2-8B_n^2=1$. In particular, $C_{2n-1}^2-8B_{2n-1}^2=1$.
\end{lemma}

\begin{proof}
This follows directly from the definition of $C_n$ (see \cite{GKP}).
\end{proof}
The following lemma gives the periodicity of the odd-indexed sequence $B_{2n-1}$ modulo $280$, which is needed to determine the correction term in Theorem 1.1.

\begin{lemma}\label{L6}
\normalfont The sequence $(B_{2n-1}\bmod 280)_{n\ge1}$ is periodic with period $12$. Moreover, $B_{2n-1}\not\equiv0\pmod{280}$ for every integer $n\ge1$, and $B_{2n-1}\equiv1\pmod{280}$
if and only if $n\equiv1\pmod{12}$.
\end{lemma}

\begin{proof}
Since $280=2^3\cdot5\cdot7$, it is enough to determine the periods of $(B_n)_{n\in \Z_{\geq 0}}$ modulo $8$, $5$, and $7$. 
The balancing numbers satisfy the recurrence $ B_{n+2}=6B_{n+1}-B_n$, with $B_0=0$, $B_1=1$. One can easily verify that the periods of $(B_n)_{n\in \Z_{\geq 0}}$ modulo $8$, $5$, and $7$ are $8$, $6$, and $3$, respectively. 
Hence, by the Chinese Remainder Theorem, we have $B_{n+24}\equiv B_n\pmod{280}$. 
Now, if we replace $n$ by $2n-1$ in the congruence, then we obtain $B_{2n-1+24} = B_{2(n+12)-1} \equiv B_{2n-1} \pmod{280}$. Hence $12$ is a period of the sequence $(B_{2n-1}\bmod 280)_{n\ge1}$. A direct computation shows that the first twelve terms of $(B_{2n-1}\bmod 280)$ are
\begin{equation*}
	1,\;35,\;69,\;71,\;105,\;139,\;141,\;175,\;209,\;211,\;245,\;279,
\end{equation*}
which are pairwise distinct. Since the first twelve residues are pairwise distinct, none of the proper divisors of $12$ can be a period. Hence, the least period is $12$.

%Therefore, none of the proper divisors of $12$
%can be a period. Hence, the least period of $(B_{2n-1}\bmod280)_{n\ge1}$ is $12$.
\end{proof}

Now, we first give a simple estimate for the growth of the balancing numbers. Using this estimate, we then obtain a bound for the tail of the reciprocal sum, which will be used later to show that the sum converges.

\begin{lemma}\label{L7}
\normalfont The balancing numbers satisfy $B_k>5B_{k-1}$, for every integer $k\ge2$.
\end{lemma}

\begin{proof}
We prove the result by induction on $k$. For $k=2$, we have $B_2 =6 > 5 = 5 \cdot B_1 $. Now assume that
$B_k>5B_{k-1}$ for some $k\ge2$. Then
\begin{equation*}
B_{k+1}=6B_k-B_{k-1}
>6B_k-\frac{1}{5}B_k
=\frac{29}{5}B_k
>5B_k.
\end{equation*}
Hence $B_{k+1}>5B_k$.
\end{proof}

\begin{corollary}\label{C2}
\normalfont For all integers $n\ge1$ and $m\ge0$,
\begin{equation}
T(n+m):=\sum_{k=n+m}^{\infty}\frac{1}{B_k^4}
<
\frac{625}{624}\cdot\frac{1}{B_{n+m}^4}.
\end{equation}
\end{corollary}

\begin{proof}
From Lemma \ref{L7}, we obtain $B_{n+m+j}>5B_{n+m+j-1}$ for every $j\ge1$. Now, by the repeated application of this inequality, we have
$B_{n+m+j}>5^{\,j}\,B_{n+m}$, for all $j \ge 1$. Hence, 
\begin{equation*}
\frac{ 1}{B_{n+m+j}^4} < \frac{1}{\,5^{4j}\,B_{n+m}^4}, \,\, \text{for} \,\, j\ge1 .
\end{equation*}
Therefore, we have the following tail bound;
\begin{equation*}
\begin{aligned} T(n+m) &=\frac{1}{B_{n+m}^4} +\sum_{j=1}^{\infty}\frac{1}{B_{n+m+j}^4} \\ 
& < \frac{1}{B_{n+m}^4} +\frac{1}{B_{n+m}^4}\sum_{j=1}^{\infty}\frac{1}{5^{4j}} \\ 
&= \frac{1}{B_{n+m}^4} \sum_{j=0}^{\infty}\frac{1}{5^{4j}} = \frac{625}{624}\cdot\frac{1}{B_{n+m}^4}. \end{aligned}
\end{equation*}
\end{proof}
\section{Key Inequalities}
In this section, we define an auxiliary quantity $g_n^{\mathrm{sm}}$ that approximates the reciprocal sum and establish a sandwiching inequality for it.  For $n\geq2$, define

\begin{equation*}
g_n^{\mathrm{sm}}
:=B_n^4-B_{n-1}^4-\frac{B_{2n-1}}{280}.
\end{equation*}
By Corollary~\ref{C1}, this can be written as
\begin{equation*}
g_n^{\mathrm{sm}}
=\frac{3B_{2n-1}\left(35C_{2n-1}-12\right)}{280}.
\end{equation*}

\begin{proposition} \label{P1}
\normalfont	For every integer $n\ge2$,
	\begin{equation*}
		g_n^{\mathrm{sm}}\,g_{n+1}^{\mathrm{sm}}
	<
	\left(g_{n+1}^{\mathrm{sm}}-g_n^{\mathrm{sm}}\right)B_n^4
	<
	\left(g_n^{\mathrm{sm}}+1\right)\left(g_{n+1}^{\mathrm{sm}}+1\right).
	\end{equation*}
\end{proposition}

\begin{proof}
To simplify the notation, let $x=B_{2n-1}$ and $y=C_{2n-1}$. From Corollary~\ref{C1}, we have
\begin{equation*}
	B_n^4=\left(\frac{8x+3y-1}{16}\right)^2.
\end{equation*}

 Further, from equation~(\ref{eq1}), we obtain $B_{2n+1}=17x+6y$ and $C_{2n+1}=48x+17y$. Hence, by the definition of $g_n^{\mathrm{sm}}$, 
\begin{equation*}
g_{n+1}^{\mathrm{sm}}
= \frac{3(17x+6y)\bigl(35(48x+17y)-12\bigr)}{280}.	
\end{equation*}
Now consider the function 

\begin{equation*}
	D_1(x,y):= \left(g_{n+1}^{\mathrm{sm}}-g_n^{\mathrm{sm}}\right)B_n^4
	-g_n^{\mathrm{sm}}g_{n+1}^{\mathrm{sm}}.
\end{equation*}
Since $g_{n+1}^{\mathrm{sm}}$, $g_{n}^{\mathrm{sm}}$ and $B_n^4$ are expressed in terms of $x$ and $y$, it follows that $D_1$ can be expressed in terms of $x$ and $y$. By Lemma~\ref{L5}, the pair $(B_{2n-1},C_{2n-1})$ satisfies the Pell equation $C_{2n-1}^2-8B_{2n-1}^2=1$, that is, $y^2=8x^2+1$. Now, we replace every occurrence of $y^2$, $y^3$ and $y^4$ by $y^2=8x^2+1$
$y^3=y(8x^2+1)$ and $y^4=(8x^2+1)^2$, respectively and using  SageMath~\cite{SageMath} to simplify the resulting expression, we obtain
\begin{equation*}
	D_1(x,y) =
	\frac{561051}{156800}x^2
	+\frac{99009}{78400}xy
	-\frac{1107}{448}x
	-\frac{3321}{3584}y
	+\frac{27099}{17920}.
\end{equation*}

Next, we want to obtain a lower bound for $D_1(x,y)$. To do this, we first establish suitable bounds for $y=\sqrt{8x^2+1}$. Since $8x^2<8x^2+1$, therefore $2\sqrt2\,x<\sqrt{8x^2+1}$. Further, we have
\begin{equation*}
	8x^2+1 < 8x^2+1+\frac{1}{32x^2} < \left(2\sqrt2\,x+\frac{1}{4\sqrt2\,x}\right)^2.
\end{equation*}
Since both sides are positive, taking square roots gives $\sqrt{8x^2+1} <
2\sqrt2\,x+\frac{1}{4\sqrt2\,x}$. Hence 
\begin{equation*}
	2\sqrt2\,x < y=\sqrt{8x^2+1} < 2\sqrt2\,x+\frac{1}{4\sqrt2\,x}.
\end{equation*}
Since the coefficient of $xy$ is positive in $D_1(x,y)$, we substitute the lower bound $ y>2\sqrt{2}\,x,$ which gives $ \frac{99009}{78400}xy > \frac{99009}{78400}x(2\sqrt{2}\,x)$. Similarly, since the coefficient of $y$ is negative, we substitute the upper bound $y<2\sqrt{2}\,x+\frac{1}{4\sqrt{2}\,x}$, which gives 
\begin{equation*}
 -\frac{3321}{3584} \left(2\sqrt{2}\,x+\frac{1}{4\sqrt{2}\,x}\right)  < -\frac{3321}{3584}y.
\end{equation*}

Therefore, we have
\begin{equation*}
\begin{aligned} 
	D_1(x,y) &> \frac{561051}{156800}x^2 +\frac{99009}{78400}x(2\sqrt{2}\,x) -\frac{1107}{448}x \\ &\qquad -\frac{3321}{3584} \left(2\sqrt{2}\,x+\frac{1}{4\sqrt{2}\,x}\right) +\frac{27099}{17920}.
\end{aligned}
\end{equation*}
After simplifying the right-hand side, we obtain
\begin{equation*}
	D_1(x,y)> \frac{N(x)}{5\,017\,600\,x},
\end{equation*}
where $N(x) = (12673152\sqrt{2} + 17953632)\,x^3
- (9298800\sqrt{2} + 12398400)\,x^2
+ 7587720\,x
- 581175\sqrt{2}$. The leading coefficient of $N'(x)$ is positive. Moreover, a direct computation using SageMath \cite{SageMath} shows that the discriminant of $N'(x)$ is negative. Therefore, $N'(x)>0$ for every $x\in\mathbb{R}$. Hence, $N(x)$ is strictly increasing on $[1,\infty)$. 
Since $ N(1)=2793177 \sqrt{2}+ 13142952>0 $, it follows that $ N(x)\ge N(1)>0$ for every $x\ge1$. As $x\geq 1$, we obtain 
\begin{equation*}
D_1(x,y)> \frac{N(x)}{5\,017\,600\,x}>0.
\end{equation*}
This proves the left inequality $g_n^{\mathrm{sm}}\,g_{n+1}^{\mathrm{sm}}
<
\left(g_{n+1}^{\mathrm{sm}}-g_n^{\mathrm{sm}}\right)B_n^4$.

For the right inequality, we define $F_1(x,y) = (g_n^{\mathrm{sm}}+1)(g_{n+1}^{\mathrm{sm}}+1) - \left(g_{n+1}^{\mathrm{sm}}-g_n^{\mathrm{sm}}\right)B_n^4.$ Now once again using the Pell equation $y^2=8x^2+1$, expanding this expression using SageMath~\cite{SageMath} gives
\begin{equation*}
	F_1(x,y) = \frac{95400549}{156800}x^2 +\frac{16894191}{78400}xy +\frac{351}{2240}x +\frac{2781}{17\,920}y +\frac{676261}{17920}.
\end{equation*}
Since $x=B_{2n-1}>0$ and $y=C_{2n-1}>0$, every coefficient in the above expression is positive. Therefore, $F_1(x,y)>0$, which implies
\begin{equation*}
	\left(g_{n+1}^{\mathrm{sm}}-g_n^{\mathrm{sm}}\right)B_n^4 < (g_n^{\mathrm{sm}}+1)(g_{n+1}^{\mathrm{sm}}+1).
\end{equation*}
\end{proof}
We now replace the smooth quantity $g_n^{\mathrm{sm}}$ by the integer quantity $g_n$ and show that the same inequalities remain true for $g_n$.

\begin{proposition} \label{P2}
 \normalfont For all $n\ge2$, we have
\begin{equation*}
	g_ng_{n+1}
	<
	(g_{n+1}-g_n)B_n^4
	<
	(g_n+1)(g_{n+1}+1).
\end{equation*}
\end{proposition}

\begin{proof}
Let $\delta_n=g_n-g_n^{\mathrm{sm}}$. From the definitions of $g_n$ and $g_n^{\mathrm{sm}}$, we obtain
\begin{equation*}
\delta_n=-\left\lceil\frac{B_{2n-1}}{280}\right\rceil
+\frac{B_{2n-1}}{280}
+\varepsilon_n.	
\end{equation*}

 By Lemma~\ref{L6}, the sequence $(B_{2n-1}~\text{mod}~280)$ has period $12$. Therefore it is enough to determine $\delta_n $ on that period. Now, using the definition of $\delta_n$ together with the corresponding residue classes of $B_{2n-1}$ modulo $280$, we obtain
 \begin{equation*}
 \delta_n\in
 \left\{
 \frac1{280},
 -\frac78,
 -\frac{211}{280},
 -\frac{209}{280},
 -\frac58,
 -\frac{141}{280},
 -\frac{139}{280},
 -\frac38,
 -\frac{71}{280},
 -\frac{69}{280},
 -\frac18,
 -\frac1{280}
 \right\}.	
 \end{equation*}

Hence, $-\frac78\le\delta_n\le\frac1{280}$. To prove the left inequality, we define
\begin{equation*}
D(x,y)
:=
(g_{n+1}-g_n)B_n^4-g_ng_{n+1},	
\end{equation*}
where $x=B_{2n-1}$ and $y=C_{2n-1}$, as in the proof of Proposition~\ref{P1}. Now, using $g_n=g_n^{\mathrm{sm}}+\delta_n$ \text{and} $g_{n+1}=g_{n+1}^{\mathrm{sm}}+\delta_{n+1}$, we obtain
\begin{align*}
	D(x,y)
	&=
	(g_{n+1}^{\mathrm{sm}}-g_n^{\mathrm{sm}})B_n^4
	-g_n^{\mathrm{sm}}g_{n+1}^{\mathrm{sm}}  \\
	&\quad
	+\delta_{n+1}B_n^4
	-\delta_nB_n^4
	-g_n^{\mathrm{sm}}\delta_{n+1}
	-g_{n+1}^{\mathrm{sm}}\delta_n
	-\delta_n\delta_{n+1} \\
	& = D_1(x,y)
	+\delta_{n+1}\bigl(B_n^4-g_n^{\mathrm{sm}}\bigr)
	-\delta_n\bigl(B_n^4+g_{n+1}^{\mathrm{sm}}\bigr)
	-\delta_n\delta_{n+1}.
\end{align*}
where $D_1(x,y)$ denotes the expression obtained in
Proposition~\ref{P1}.

The expression $D(x,y)$ is a bilinear map in the variables $\delta_n$ and $\delta_{n+1}$. Indeed, for any fixed value of $\delta_{n+1}$, the expression $D(x,y)$ is linear in $\delta_n$, and for any fixed value of $\delta_n$, it is linear in $\delta_{n+1}$. Therefore, for any fixed $\delta_{n+1}$, the minimum of $D(x,y)$ with respect to $\delta_n$ over the interval $-\frac78\le\delta_n\le\frac1{280}$ is attained at one of the two endpoints, namely $\delta_n=-\frac78$ or $\delta_n=\frac1{280}$. The same argument applies when $\delta_n$ is fixed. In this case, $D(x,y)$ is a linear function of $\delta_{n+1}$, so its minimum on the interval $-\frac78\le\delta_{n+1}\le\frac1{280}$ is attained at one of the two endpoints. Therefore, the minimum value of $D(x,y)$ over the rectangle $-\frac78\le\delta_n,\delta_{n+1}\le\frac1{280}$ is attained at one of the four corner points
\begin{equation*}
	\left(-\frac78,-\frac78\right),\,\,
	\left(-\frac78,\frac1{280}\right),\,\,
	\left(\frac1{280},-\frac78\right),\,\,
	\left(\frac1{280},\frac1{280}\right).
\end{equation*}
Hence, it is enough to show that $D(x,y) > 0 $ for at each of the four corner points. Now again expanding the expression of $D(x,y)$ in SageMath~\cite{SageMath} gives
\begin{equation}\label{eq5}
	\begin{aligned}
		D(x,y) = {} & D_1(x,y) \\
		&  + \delta_{n+1}\left(\frac{17}{32}x^2 - \frac{3}{16}xy + \frac{37}{560}x - \frac{3}{128}y + \frac{5}{128}\right) \\
		& - \delta_n\left(\frac{19601}{32}x^2 + \frac{3465}{16}xy - \frac{1259}{560}x - \frac{3561}{4480}y + \frac{4901}{128}\right) \\
		& - \delta_n\delta_{n+1}.
	\end{aligned}
\end{equation}
Using the relation $y^2=8x^2+1$, equation~\eqref{eq5}
further simplifies to
\begin{equation}\label{eq6}
	\begin{aligned}
		D(x,y) = {} & D_1(x,y) \\
		& - \frac{19601}{32}x^2\delta_n - \frac{3465}{16}xy\,\delta_n + \frac{1259}{560}x\,\delta_n + \frac{3561}{4480}y\,\delta_n - \frac{4901}{128}\delta_n \\
		& + \frac{17}{32}x^2\delta_{n+1} - \frac{3}{16}xy\,\delta_{n+1} + \frac{37}{560}x\,\delta_{n+1} - \frac{3}{128}y\,\delta_{n+1} + \frac{5}{128}\delta_{n+1} \\
		& - \delta_n\delta_{n+1}.
	\end{aligned}	
\end{equation}
We now verify that $D(x,y)>0$ at each of the four corner
points. Recall that $x=B_{2n-1}\geq B_3=35$ and
$y=C_{2n-1}=\sqrt{8x^2+1}>0$ for all $n\geq2$. First consider $\left(\delta_n,\delta_{n+1}\right)
=
\left(-\frac78,-\frac78\right)$. Now, after substituting these values into \eqref{eq6}, we obtain 
\begin{equation*}
D(x,y)=P_1(x)+Q_1(x)y,
\end{equation*}
where
\begin{equation*}
P_1(x)
=
\frac{84527451}{156800}x^2
-\frac{10071}{2240}x
+\frac{613139}{17920} \,\, \text{and} \,\, Q_1(x)
=
\frac{14968059}{78400}x
-\frac{28701}{17920}.
\end{equation*}
For $x\geq35$, we have $Q_1(x)>0$, as $Q_1$ is an increasing linear function and $Q_1(35)>0$. Moreover,
\begin{equation*}
P_1'(x)
=
\frac{84527451}{78400}x-\frac{10071}{2240}.
\end{equation*}

$P_1'$ is an increasing linear function and $P_1'(35)>0$. Therefore, $P_1'(x)>0$ for every $x\ge35$. Hence, $P_1$ is strictly increasing on $[35,\infty)$. Also, $P_1(35)>0$, so $P_1(x)>0$ for every $x\ge35$. As $y>0$, it follows that $D(x,y)=P_1(x)+Q_1(x)y>0$. 

Next consider $\left(\delta_n,\delta_{n+1}\right)=\left(-\frac78,\frac1{280}\right)$. After substituting these values into \eqref{eq6}, we get 
\begin{equation*}
	D(x,y)=P_2(x)+Q_2(x)y,
\end{equation*}
where
\begin{equation*}
P_2(x)
=
\frac{21150159}{39200}x^2
-\frac{173967}{39200}x
+\frac{62753}{1792}, \,\, \text{and} \,\, Q_2(x)
=
\frac{1869393}{9800}x
-\frac{2907}{1792}.
\end{equation*}
Again, $Q_2$ is an increasing linear function and
$Q_2(35)>0$. Moreover,
\begin{equation*}
P_2'(x)
=
\frac{21150159}{19600}x-\frac{173967}{39200}>0
\end{equation*}
for $x\geq35$, and $P_2(35)>0$. Therefore $P_2(x)>0,\,\, Q_2(x)>0$
for all $x\geq35$. Consequently, $D(x,y)>0$. 

For the third corner, $\left(\delta_n,\delta_{n+1}\right)
=
\left(\frac1{280},-\frac78\right)$, after substituting these values into \eqref{eq6}, we obtain
\begin{equation*}
D(x,y)=P_3(x)+Q_3(x)y,
\end{equation*}
where 
\begin{equation*}
P_3(x)
=
\frac{72573}{78400}x^2
-\frac{49407}{19600}x
+\frac{6023}{4480},\,\, \text{and} \,\,	Q_3(x)
=
\frac{25617}{39200}x
-\frac{141633}{156800}.
\end{equation*}
The function $Q_3$ is increasing and $Q_3(35)>0$. Also,
\begin{equation*}
P_3'(x)
=
\frac{72573}{39200}x-\frac{49407}{19600}>0
\end{equation*}
for $x\geq35$, and $P_3(35)>0$. Thus $P_3(x)>0$, $Q_3(x)>0$ for all $x\geq35$, and hence $D(x,y)>0$. 

Finally, consider the point $\left(\delta_n,\delta_{n+1}\right)
=
\left(\frac1{280},\frac1{280}\right)$, after substituting these values into \eqref{eq6}, we obtain
\begin{equation*}
	D(x,y)=P_4(x)+Q_4(x)y,
\end{equation*}
where
\begin{equation*}
	P_4(x)
	=
	\frac{218331}{156800}x^2
	-\frac{193077}{78400}x
	+\frac{862777}{627200}, \,\, \text{and} \,\, Q_4(x)
	=
	\frac{38319}{78400}x
	-\frac{579447}{627200}.
\end{equation*}
The function $Q_4$ is increasing and $Q_4(35)>0$. Furthermore,
\begin{equation*}
P_4'(x)
=
\frac{218331}{78400}x-\frac{193077}{78400}>0
\end{equation*}
for $x\geq35$, and $P_4(35)>0$. Hence $P_4(x)>0$, $Q_4(x)>0$ for all $x\geq35$. Therefore, $D(x,y)>0$. We have thus proved that $D(x,y)>0$ at all four corner points of the rectangle $\left[-\frac78,\frac1{280}\right]^2$. Since $D(x,y)$ is bilinear in $(\delta_n,\delta_{n+1})$, its
minimum on this rectangle is attained at one of these four
corners. It follows that $D(x,y)>0$ for every admissible pair $(\delta_n,\delta_{n+1})$. Therefore, $g_ng_{n+1} < (g_{n+1}-g_n)B_n^4$,
which proves the left-hand inequality in Proposition~\ref{P2}.

Next, for the right inequality, we define 
\begin{equation*}
F(x,y)
	:=
(g_n+1)(g_{n+1}+1)
-
(g_{n+1}-g_n)B_n^4,
\end{equation*}
where $x=B_{2n-1}$ and $y=C_{2n-1}$. From the relations $g_n=g_n^{\mathrm{sm}}+\delta_n$ and  $g_{n+1}=g_{n+1}^{\mathrm{sm}}+\delta_{n+1}$, we obtain
\begin{align*}
	F(x,y)
	={}&
	(g_n^{\mathrm{sm}}+1)(g_{n+1}^{\mathrm{sm}}+1)
	-
	(g_{n+1}^{\mathrm{sm}}-g_n^{\mathrm{sm}})B_n^4
	\\
	&+\delta_n(g_{n+1}^{\mathrm{sm}}+1)
	+\delta_{n+1}(g_n^{\mathrm{sm}}+1)
	+\delta_n\delta_{n+1} -\delta_{n+1}B_n^4+\delta_nB_n^4. \\
	& = F_1(x,y) + \delta_{n+1}\left(g_n^{sm}+1-B_n^4\right) + \delta_n\left(g_{n+1}^{sm}+1+B_n^4\right) + \delta_n\delta_{n+1}
\end{align*}
As above, $F(x,y)$ is bilinear in $(\delta_n,\delta_{n+1})$.  Hence, by the same  argument, it suffices to verify that $F(x,y)>0$ at the four corners $\left(-\frac78,-\frac78\right)$, 
$\left(-\frac78,\frac1{280}\right)$,$
\left(\frac1{280},-\frac78\right)$, and $
\left(\frac1{280},\frac1{280}\right)$. 
For the first corner $(\delta_n,\delta_{n+1})=
\left(-\frac78,-\frac78\right)$,
after substituting these values into the expression for $F(x,y)$
and simplifying, we obtain
\begin{equation*}
F(x,y)=G_1(x)+H_1(x)y,
\end{equation*}
where
\begin{equation*}
G_1(x)
=\frac{11434149}{156800}x^2
+\frac{4887}{2240}x
+\frac{58861}{17920} \,\, \text{and} \,\, H_1(x)
=
\frac{2025141}{78400}x
+\frac{14877}{17920}.
\end{equation*}
Since $x,\,y>0$ and all coefficients of $G_1(x)$ and $H_1(x)$ are positive, we have  $F(x,y)=G_1(x)+H_1(x)y>0$. For the second corner
$(\delta_n,\delta_{n+1})=\left(-\frac78,\frac1{280}\right)$, after substitution and simplification gives
\begin{equation*}
F(x,y)=G_2(x)+H_2(x)y,
\end{equation*}
where
\begin{equation*}
G_2(x)
=
\frac{2840241}{39200}x^2
+\frac{83247}{39200}x
+\frac{4301}{1280}, \,\, \text{and}\,\, H_2(x)
=
\frac{254757}{9800}x
+\frac{1089}{1280}.
\end{equation*}
Again, all coefficients of $G_2(x)$ and $H_2(x)$ are positive.
Therefore, $F(x,y)>0$. For the third corner $(\delta_n,\delta_{n+1}) = \left(\frac1{280},-\frac78\right)$, after substitution and simplification
gives 
\begin{equation*}
	F(x,y)=G_3(x)+H_3(x)y,
\end{equation*}
where 
\begin{equation*}
G_3(x)
=
\frac{47908227}{78400}x^2
+\frac{4047}{19600}x
+\frac{165913}{4480}, \,\, \text{and} \,\, 	H_3(x)
=
\frac{8470983}{39200}x
+\frac{20673}{156800}.
\end{equation*}
All coefficients in these two expressions are positive. Hence
$F(x,y)>0$ for $x,y>0$. Finally, consider the fourth corner point $(\delta_n,\delta_{n+1})
= \left(\frac1{280},\frac1{280}\right)$, after simplication, we get
\begin{equation*}
F(x,y)=G_4(x)+H_4(x)y,
\end{equation*}
where
\begin{equation*}
G_4(x)
=
\frac{95743269}{156800}x^2
+\frac{11637}{78400}x
+\frac{23759303}{627200}, \,\, \text{and} \,\, H_4(x)
=
\frac{16954881}{78400}x
+\frac{95607}{627200}.
\end{equation*}
All coefficients of $G_4(x)$ and $H_4(x)$ are positive. Thus $F(x,y)>0$. We have proved that $F(x,y)>0$ at all four corners points. Since its minimum on this rectangle is attained at one of these four corners. Therefore, $F(x,y)>0$. Now, it follows from the definition of $F(x,y)$ that
\begin{equation*}
(g_n+1)(g_{n+1}+1)
-
(g_{n+1}-g_n)B_n^4>0.
\end{equation*}
Hence,
\[
(g_{n+1}-g_n)B_n^4
<
(g_n+1)(g_{n+1}+1),
\]
which proves the right-hand inequality.
\end{proof}

\section{Proof of Main Theorem}

We are now ready to prove the main result of this paper. It gives a closed formula for the integer part of the reciprocal of the infinite tail sum of the fourth power of the balancing numbers.
%\begin{theorem}
%\normalfont Let $(B_n)_{n\ge0}$ be the sequence of balancing numbers, and let $g_n$ be defined as above. Then, for every integer $n\ge2$, the reciprocal of the tail sum is
%\begin{equation*}
%	\left\lfloor
	%\left(
%	\sum_{k=n}^{\infty}\frac{1}{B_k^4}
%	\right)^{-1}
%	\right\rfloor
%	=
%	g_n.
%\end{equation*}
%\end{theorem}

\begin{proof}(Proof of Theorem \ref{A})\\
In order to prove the main theorem, it is enough to prove that
\begin{equation*}\label{eq:mainineq}
	g_n\le T(n)^{-1}<g_n+1
\end{equation*} for every integer $n\ge2$,
where
\begin{equation*}
T(n)=\sum_{k=n}^{\infty}\frac{1}{B_k^4}.
\end{equation*}
 By Proposition~\ref{P2}, for every integer $n\ge2$, we have
\begin{equation*}
g_ng_{n+1}<(g_{n+1}-g_n)B_n^{4}
<(g_n+1)(g_{n+1}+1).	
\end{equation*}
Since $g_n>0$ and $g_{n+1}>0$, dividing the left inequality by $g_ng_{n+1}\,B_n^{4}$ gives
\begin{equation*}
	\frac{1}{B_n^{4}}+\frac{1}{g_{n+1}}
	< \frac{1}{g_n},
\end{equation*}
while dividing the right inequality by $(g_n+1)(g_{n+1}+1)B_n^{4}$ yields
\begin{equation*}
	\frac{1}{g_n+1}
	<
	\frac{1}{B_n^{4}}
	+
	\frac{1}{g_{n+1}+1}.
\end{equation*}
For any integer $m\ge n$, repeated application of the first inequality gives
\begin{equation}\label{eq3}
  \sum_{k=n}^{m}\frac{1}{B_k^{4}}
  +\frac{1}{g_{m+1}} \le	\frac{1}{g_n}.
\end{equation}
Similarly, If we apply the second inequality repeatedly, then
\begin{equation}\label{eq4}
	\begin{aligned}
		\frac{1}{g_n+1} <
		\sum_{k=n}^{m}\frac{1}{B_k^{4}}
		+\frac{1}{g_{m+1}+1}.
	\end{aligned}
\end{equation}
Since
\begin{equation*}
g_n=B_n^{4}-B_{n-1}^{4}-\left\lfloor\frac{B_{2n-1}^{2}}{280}\right\rfloor+\varepsilon_n,
\end{equation*}
we have $g_n\to\infty$ as $n\to\infty$. Now, letting  $m\to\infty$ in the above equations (\ref{eq3}) and (\ref{eq4}) gives
\begin{equation*}
\frac{1}{g_n+1} < \,\sum_{k=n}^{\infty}\frac1{B_k^{4}} \, \le \frac1{g_n},
\end{equation*}
that is, $\frac{1}{g_n+1}<T(n)\le\frac1{g_n}$. Finally, taking reciprocals  gives
$g_n\le T(n)^{-1}<g_n+1$. This completes the proof of the theorem. 
\end{proof}

\section*{Declarations}
\textbf{Conflict of interest:} 
The authors declare that they do not have conflict of interests. \vspace{0.2cm}\\
\textbf{Funding:}
No funding was received.

%\section*{Acknowledgement}
%We would like to thank the editor and the anonymous referees for their valuable comments and helpful feedback.

\end{document}